\documentclass[12pt]{amsart}  % draft

\usepackage{custom_preamble,mathtools} % showkeys

\begin{document}

\title{On inequivalences of sequences of characters}

\author{\tsname}
\address{\tsaddress}
\email{\tsemail}

\maketitle

\begin{abstract}
We establish various results including the following: if $1<p<\infty$ is not equal to $2$ and $\sigma$ is a bijection between the trigonometric functions on $[0,1)$ and the Walsh functions on $[0,1)$.  Then $\sigma$ does \emph{not} extend to an isomorphism $L_p[0,1) \rightarrow L_p[0,1)$.
\end{abstract}

\section{Introduction}

In the papers \cite{woj::0} and \cite{czuwoj::}, Wojciechowski and then Czuron and Wojciechowski applied particular tools from additive combinatorics and number theory to problems about the equivalence of various systems of characters on compact Abelian groups.  In this note we aim to build on their success by extracting the key ingredients from these tools and applying them directly.

The \textbf{trigonometric functions} are the functions $e_z:[0,1) \rightarrow \C; x \mapsto \exp(2\pi i zx)$ for $z \in \Z$, and the \textbf{Walsh functions} are the functions $w_J:[0,1) \rightarrow \C; x \mapsto\prod_{j \in J}{r_j(x)}$ where $J \subset \N_0$ is finite and for $j \in \N_0$, $r_j:[0,1) \rightarrow \{-1,1\}$ is the \textbf{$j$th Radamacher function}, defined by $r_j(x)=1$ if and only if $2^jx-\lfloor 2^jx\rfloor <\frac{1}{2}$.

Both the trigonometric and the Walsh functions are orthonormal bases for $L_2[0,1)$, and any bijection between them extends to an (isometric) isomorphism $L_2[0,1) \rightarrow L_2[0,1)$.  In \cite[\S3.2]{woj::1} and \cite[Problem 5.3]{pel::1} the question of what happens for $1 <p < \infty$ when $p \neq 2$ is asked, and we answer this as follows.
\begin{theorem}\label{thm.pel}
Suppose that $1 <p < \infty$ with $p\neq 2$ is a parameter and $\sigma$ is a bijection between the trigonometric and Walsh functions.  Then $\sigma$ does \emph{not} extend to an isomorphism $L_p[0,1) \rightarrow L_p[0,1)$.
\end{theorem}
Young \cite[Theorem 2]{you::0} proved the above for one of the most natural candidates for $\sigma$, taking the ordering $e_0,e_1,e_{-1},e_2,e_{-2},\dots$ of the trigonometric functions to the Walsh-Paley ordering $1, r_1,r_2,r_1r_2,r_3,r_1r_2,r_2r_3,r_1r_2r_3,\dots $ of the Walsh functions.

In \cite{hinwen::0}, Hinrichs and Wenzel proved Theorem \ref{thm.pel} for a wide range of bijections including those given by replacing the Walsh-Paley ordering by the Walsh-Kronecker and Walsh-Kaczmarz orderings.  Our work covers the remaining possibilities.

Theorem \ref{thm.pel} does not extend to bijections between arbitrary infinite sets of trigonometric and Walsh functions, even if we only ask for an isomorphism between the closures of their respective linear spans.  To see why consider 
\begin{equation*}
\sigma:\{e_{3^{n}}:n \in \N_0\} \rightarrow \{r_n: n \in \N_0\}; e_{3^{n}} \mapsto r_n \text{ for all }n \in \N_0.
\end{equation*}
For $1\leq p<\infty$ we have a composition of isomorphisms
\begin{equation*}
\overline{\Span(e_{3^{n}}: n \in \N_0)} \rightarrow \ell_2 \rightarrow \overline{\Span(r_n: n \in \N_0)}
\end{equation*}
extending $e_{3^{n}}\mapsto \delta_n$ and $\delta_n \mapsto r_n$ (where $(\delta_n)_{n \in \N_0}$ is the standard basis for $\ell_2$) by \cite[(8.20), p215]{zyg::} and \cite[(8.4), p213]{zyg::}.  It follows that $\sigma$ \emph{does} extend to an isomorphism between the closures of the respective linear spans.

On the other hand Theorem \ref{thm.pel} does remain true for the endpoints $p=1$ or $p=\infty$ as noted in the discussion after \cite[Corollary 1.5]{woj::0} (indeed, it is also a consequence of that corollary).

The trigonometric and Walsh functions are naturally related to characters on compact Abelian groups; we shall take the book \cite{rud::1} as our basic reference.

Suppose that $G$ is a compact Abelian group. We write $\wh{G}$ for the dual group of $G$, that is the discrete Abelian group of continuous homomorphisms $G \rightarrow S^1$ -- the characters of $G$ -- where $S^1:=\{z \in \Z: |z|=1\}$.  We also write $G^\omega$ for the set of all functions $\N_0 \rightarrow G$ under point-wise operations inherited from $G$ and endowed with the product topology.

The trigonometric functions may be identified with the characters on the group $\T:=\R/\Z$ and the Walsh functions may be identified with the characters on $\D:=(\Z/2\Z)^{\omega}$.  We shall return to the details of this later, but before that we give a second application of the basic method to \cite[Theorem 1]{czuwoj::}.

Suppose that $G$ and $H$ are compact Abelian groups.  Any continuous surjective (group) homomorphism $G \rightarrow H$ pulls back to an algebra homomorphism, isometric onto its image,
\begin{equation}\label{eqn.isom}
L_1(H) \rightarrow L_1(G); f \mapsto (x \mapsto f(\phi(x))),
\end{equation}
where the algebra multiplication on $L_1$ is convolution with respect to Haar probability measure.

On the other hand, not all maps of the latter type arise as pullbacks of the former: suppose $G:=\D$ and $H:=\Z/p\Z$ for $p$ an odd prime.  Since $G$ is infinite so is $\wh{G}$, and so there is an injection $\phi:\wh{H} \rightarrow \wh{G}$.  Then
\begin{equation}\label{eqn.cest}
L_1(H) \rightarrow L_1(G); f  \mapsto \sum_{\lambda \in \wh{H}}{\wh{f}(\lambda)\phi(\lambda)}
\end{equation}
is a continuous injective algebra homomorphism.  However, since $H$ has odd order and all elements of $G$ have even order, $H$ cannot arise as a quotient of $G$.

We are interested in some cases where such maps do not arise. 
\begin{theorem}\label{thm.main}
Suppose that $G$ is a compact Abelian torsion group and $H$ is a finite Abelian group such that there is no continuous surjective homomorphism $G \rightarrow H^\omega$.  Then there is $N_{G,H}\in \N$ such that for all naturals $n \geq N_{G,H}$ and continuous injective algebra homomorphisms $T:L_1(H^n) \rightarrow L_1(G)$ we have $\|T\| \geq n^{\frac{1}{4}-o_{n \rightarrow \infty}(1)}$.  In particular, there is no continuous injective algebra homomorphism $L_1(H^\omega) \rightarrow L_1(G)$.
\end{theorem}
Computing the norm of the map in (\ref{eqn.cest}), we see that if $G$ is infinite and $n \in \N$ then there is a continuous injective algebra homomorphism $T:L_1(H^n) \rightarrow L_1(G)$ such that $\|T\| \leq |H|^n$ which gives a limit to how far the theorem may be improved.

Our contribution in Theorem \ref{thm.main} is in the bounds.  \cite[Theorem 1]{czuwoj::} establishes roughly the same qualitative conclusion for $H$ a cyclic group of prime-power order but with the hypothesis that no subgroup of $G$ is homeomorphically isomorphic to $H^\omega$ instead of that there is no continuous surjective homomorphism $G \rightarrow H^\omega$.  One can establish the same quantitative conclusion as in Theorem \ref{thm.main} with the hypotheses of \cite[Theorem 1]{czuwoj::} by noting that if the hypotheses of \cite[Theorem 1]{czuwoj::} hold then there cannot be a continuous surjective homomorphism $G \rightarrow H^\omega$ because an isometry as in (\ref{eqn.isom}) cannot exist; we can therefore apply Theorem \ref{thm.main}.

To keep the paper readable we have included statements of the results from additive combinatorics which we use on the basis that these are the more unusual tools.  In the more routine analytic parts of the arguments we use various results from the books of Rudin \cite{rud::1} and Wojtaszczyk \cite{woj::} which we have not stated separately though precise references are given.

\section{The combinatorial content of the main theorems}\label{sec.tools}

Our arguments take advantage of the differing behaviour of the action of the integers on different groups.  For $Z$ an Abelian group, $m \in \N$ and $x \in Z$ define $m\cdot x := x+(m-1)\cdot x$ and $0\cdot x :=0_Z$, and for $S \subset Z$ put $m\cdot S:=\{m\cdot s: s \in S\}$.  (Compare this with $mS$ which is defined by $mS = S+(m-1)S$ for $m \in \N$ and $0S=\{0_Z\}$.)  We call the map $Z \rightarrow Z; x \mapsto m\cdot x$ multiplication by $m$.

Pl{\"u}nnecke's inequality is the following basic inequality about growth of sumsets.  It will be an important tool for us.
\begin{theorem}[Pl{\"u}nnecke's Inequality, {\cite[Corollary 6.26]{taovu::}}]
Suppose that $X$ is a finite subset of an Abelian group with $|X+X| \leq K|X|$.  Then for any $k \in \N$ we have $|kX| \leq K^k|X|$.
\end{theorem}
There is a lovely proof of this due to Petridis \cite{pet::} which the interested reader is encouraged to consult.

We capture the common combinatorial element of Theorems \ref{thm.pel} \& \ref{thm.main} in the following proposition.
\begin{proposition}\label{prop.k}
Suppose that $G$ and $H$ are compact Abelian groups; $\Lambda \subset \wh{G}$ is finite; $X \subset \wh{G} \times \wh{H}$ is a subset of the graph of a function $\phi:\Lambda \rightarrow \wh{H}$; and $m \in \N$ is a parameter.  Then 
\begin{equation*}
|X+X| \geq \left(\frac{|m \cdot X|}{|m\cdot\wh{G}|}\right)^{\frac{1}{m+1}}|X|.
\end{equation*}
\end{proposition}
\begin{proof}
Let $K$ be such that $|X+X|=K|X|$.  By Pl{\"u}nnecke's inequality we have $|X + m\cdot X| \leq |(m+1)X| \leq K^{m+1}|X|$, and so by Cauchy-Schwarz we have
\begin{equation}\label{eqn.above}
\|1_{X} \ast 1_{m\cdot X}\|_{\ell_2(\wh{G}\times \wh{H})}^2 \geq \frac{\|1_{X} \ast 1_{m\cdot X}\|_{\ell_1(\wh{G}\times \wh{H})}^2}{|X+m\cdot X|} = \frac{|X|^2|m\cdot X|^2}{|X+m\cdot X|}\geq K^{-(m+1)}|X||m\cdot X|^2.
\end{equation}
The left hand side is the number of quadruples $((\lambda,\phi(\lambda)),m\cdot ( \gamma,\phi(\gamma)),(\lambda',\phi(\lambda')),m\cdot ( \gamma',\phi(\gamma')))$ such that
\begin{equation*}
(\lambda,\phi(\lambda))+m\cdot(\gamma, \phi(\gamma))=(\lambda',\phi(\lambda')) + m\cdot (\gamma',\phi(\gamma')).
\end{equation*}
It follows that such a quadruple is uniquely determined by $(\lambda, \lambda-\lambda', m\cdot (\gamma,\phi(\gamma)))$.  But $\lambda-\lambda' \in m\cdot \wh{G}$ and so there are at most $|X||m\cdot \wh{G}||m\cdot X|$ such triples, which is then an upper bound for (\ref{eqn.above}).  Cancelling and rearranging gives the result.
\end{proof}
At first glance it may seem like there is considerable loss in saying $|X+m\cdot X| \leq |(m+1)X|$.  This is true as far as the $m$-dependence goes (as shown by Bukh in \cite{buk::0}) but this will not concern us.

\section{Proof of Theorem \ref{thm.pel}}

The \textbf{additive energy} \cite[\S2.3]{taovu::} of a finite subset $\Lambda$ of an Abelian group is
\begin{equation*}
E(\Lambda) :=\sum_{\lambda_1+\lambda_2=\lambda_3+\lambda_4}{1_\Lambda(\lambda_1)1_\Lambda(\lambda_2)1_\Lambda(\lambda_3)1_\Lambda(\lambda_4)}.
\end{equation*}
This is an important quantity in additive combinatorics and its relevance for the problems considered here is identified in \cite[\S4]{hinwen::0}.

We shall be interested in graphs of functions and shall use the Balog-Szemer{\'e}di-Gowers Theorem to study those graphs with large additive energy.  This idea goes back to the celebrated work of Gowers \cite{gow::4} on Szemer{\'e}di's Theorem.
\begin{theorem}[Balog-Szemer{\'e}di-Gowers Theorem, {\cite[Theorem 2.31]{taovu::}}]
Suppose that $\Lambda$ is finite subset of an Abelian group with $E(\Lambda) \geq c|\Lambda|^3$.  Then there is a set $\Lambda' \subset \Lambda$ with $|\Lambda'| \geq c^{O(1)}|\Lambda|$ and $|\Lambda'+\Lambda'| \leq c^{-O(1)}|\Lambda'|$.
\end{theorem}
There is a proof of this in \cite{gow::4} and for the best-known constants behind the big-$O$ terms see \cite{sch::7}, though we shall not be concerned with these.

To state the driver behind the proof of Theorem \ref{thm.pel} we need one further piece of notation.  Given $G$, a compact Abelian group, and $\Lambda \subset \wh{G}$ finite, we define the projection
\begin{equation*}
\pi_\Lambda:M(G) \rightarrow L_1(G); \mu \mapsto \sum_{\lambda \in \Lambda}{\wh{\mu}(\lambda)\lambda},
\end{equation*}
and write $L_p^\Lambda(G)$ for the (Banach-)subspace of $f \in L_p(G)$ with $\supp \wh{f} \subset \Lambda$ (which is of course finite dimensional when $\Lambda$ is finite).
\begin{proposition}\label{prop.main}
Suppose that $G$ and $H$ are compact Abelian groups; $1 \leq p <2$ is a parameter; $\Lambda \subset \wh{G}$ is finite with $\|\pi_\Lambda\| \leq K$; $T:L_p^\Lambda(G) \rightarrow L_p(H)$ is a linear mapping taking characters to characters; and $m\in \N$ is a parameter such that multiplication by $m$ is injective on $T(\Lambda)$.  Then
\begin{equation*}
\|T\|_{p \rightarrow p} \geq \left(\frac{|T(\Lambda)|}{|m \cdot \wh{G}|}\right)^{\Omega\left(\frac{2-p}{m}\right)}K^{-O(2-p)}.
\end{equation*}
\end{proposition}
\begin{proof}
Let $g:=\sum_{\lambda \in \Lambda}{\lambda} \in C(G)$ and $\mu \in M(G)$ have unit norm such that $\langle \mu,g\rangle = \|g\|_{L_\infty(G)}=|\Lambda|$.  Then $f:=\pi_\Lambda(\mu)$ has $\langle f,g\rangle = |\Lambda|$ and $\|f\|_{L_1(G)} \leq K$.

By $\log$-convexity of $L_q$-norms, if $\phi$ is such that $\frac{1}{p} = \frac{\phi}{2} + (1-\phi)$ then
\begin{equation*}
\|f\|_{L_p(G)} \leq  \|f\|_{L_2(G)}^{\phi} \|f\|_{L_1(G)}^{1-\phi} \leq \|f\|_{L_2(G)}^{2\left(1-\frac{1}{p}\right)} K^{\frac{2}{p}-1}.
\end{equation*}
For $x,z \in G$ and $h:G \rightarrow \C$ write $\tau_z(h)(x):=h(x+z)$.  By $\log$-convexity of $L_q$-norms, if $\theta$ is such that $\frac{1}{2} = \frac{(1-\theta)}{p} + \frac{\theta}{4}$ (so $\theta = \frac{4-2p}{4-p}$) then
\begin{equation*}
\|T(\tau_z(f))\|_{L_p(H)}^{1-\theta}\|T(\tau_z(f))\|_{L_4(H)}^\theta \geq \|T(\tau_z(f))\|_{L_2(H)}.
\end{equation*}
Hence
\begin{equation*}
\|T(\tau_z(f))\|_{L_4(H)}^4 \geq \frac{\|T(\tau_z(f))\|_{L_2(H)}^{4\theta^{-1}}}{\|T\|_{p \rightarrow p}^{4(\theta^{-1}-1)}\|f\|_{L_p(G)}^{4(\theta^{-1}-1)}} = \frac{\|T(\tau_z(f))\|_{L_2(H)}^{\frac{8-2p}{2-p}}}{\|T\|_{p \rightarrow p}^{\frac{2p}{2-p}}\|f\|_{L_p(G)}^{\frac{2p}{2-p}}} \geq \frac{\|T(\tau_z(f))\|_{L_2(H)}^{\frac{8-2p}{2-p}}}{\|T\|_{p \rightarrow p}^{\frac{2p}{2-p}}K^2\|f\|_{L_2(G)}^{\frac{4p-4}{2-p}}}.
\end{equation*}
By convexity of $X \mapsto X^{\frac{4-p}{2-p}}$ and Jensen's inequality; linearity of $T$; interchanging the order of integration; orthogonality of characters; and the fact $T$ takes characters to unit vectors in $L_2$, we have
\begin{align*}
\int{\|T(\tau_z(f))\|_{L_2(H)}^{\frac{8-2p}{2-p}}dm_G(z)} & \geq \left(\int{\|T(\tau_z(f))\|_{L_2(H)}^2dm_G(z)} \right)^{\frac{4-p}{2-p}}\\ & = \left(\int{\left|\sum_{\lambda \in \Lambda}{\wh{f}(\lambda)\lambda(z)T(\lambda)(w)}\right|^2dm_H(w)dm_G(z)} \right)^{\frac{4-p}{2-p}}\\ & = \left(\sum_{\lambda \in \Lambda}{|\wh{f}(\lambda)|^2} \right)^{\frac{4-p}{2-p}}=\|f\|_{L_2(G)}^{\frac{8-2p}{2-p}},
\end{align*}
and so
\begin{equation*}
\int{\|T(\tau_z(f))\|_{L_4(H)}^4dm_G(z)} \geq \frac{\|f\|_{L_2(G)}^6}{\|T\|_{p \rightarrow p}^{\frac{2p}{2-p}}K^2}.
\end{equation*}
On the other hand
\begin{align*}
\int{\|T(\tau_z(f))\|_{L_4(H)}^4dm_G(z)} & = \int{\left|\sum_{\lambda \in \Lambda}{\wh{f}(\lambda)\lambda(z)T(\lambda)(w)}\right|^4dm_H(w)dm_G(z)}\\
& = \sum_{\substack{\lambda_1+\lambda_4=\lambda_2+\lambda_3\\ T(\lambda_1)+T(\lambda_4)=T(\lambda_2)+T(\lambda_3)}}{\wh{f}(\lambda_1)\overline{\wh{f}(\lambda_2)\wh{f}(\lambda_3)}\wh{f}(\lambda_4)}\\
& \leq \|f\|_{L_1(G)}^4E(\Gamma) \leq K^4E(\Gamma),
\end{align*}
where $\Gamma$ is the graph of $\Lambda \rightarrow \wh{H}; \lambda \mapsto T(\lambda)$ in $\wh{G} \times \wh{H}$.  Combining the above with Cauchy-Schwarz we get
\begin{align*}
E(\Gamma)  \geq \|f\|_{L_2(G)}^6\cdot \frac{1}{\|T\|_{p \rightarrow p}^{\frac{2p}{2-p}}K^6} & \geq \left(\frac{|\langle f,g\rangle|}{\|g\|_{L_2(G)}}\right)^6\cdot \frac{1}{\|T\|_{p \rightarrow p}^{\frac{2p}{2-p}}K^6}\\ & \geq |\Lambda|^3\cdot \frac{1}{\|T\|_{p \rightarrow p}^{\frac{2p}{2-p}}K^6} = |\Gamma|^3\cdot \frac{1}{\|T\|_{p \rightarrow p}^{\frac{2p}{2-p}}K^6}.
\end{align*}
Finally, let $\epsilon:=|\Gamma|^{-3}E(\Gamma)$ so that by the above $\epsilon^{-1}\leq \|T\|_{p \rightarrow p}^{\frac{2p}{2-p}}K^6$.  Apply the Balog-Szemer{\'e}di-Gowers Theorem to $\Gamma$ to get $X \subset \Gamma$ with
\begin{equation*}
|X+X|\leq  \|T\|_{p \rightarrow p}^{\frac{O(1)}{2-p}}K^{O(1)}|X| \text{ and } |X| \geq \|T\|_{p \rightarrow p}^{-\frac{O(1)}{2-p}}K^{-O(1)}|\Gamma|.
\end{equation*}
We now apply Proposition \ref{prop.k} and use that $|m\cdot X| \geq \|T\|_{p \rightarrow p}^{-\frac{O(1)}{2-p}}K^{-O(1)} |T(\Lambda)|$ since multiplication by $m$ is injective on $T(\Lambda)$ and $|T(\Lambda)| \leq |\Gamma|$.  This rearranges to give the result.
\end{proof}
The start of the argument above is similar to the proof of \cite[Proposition 3.1]{hinwen::0}, though for us the bound $\|\pi_\Lambda\| \leq K$ ensures that a duality argument provides functions which take on the role of the Dirichlet kernel used in \cite[Proposition 3.1]{hinwen::0}.

Proposition \ref{prop.main} immediately answers a stronger local\footnote{For a discussion of what local means in Banach spaces see \cite{pie::}.} version of \cite[Problem 5.3]{pel::1} discussed in the text following that statement.
\begin{corollary}\label{cor.1}
Suppose that $1 \leq p <2$ and $T:L_p((\Z/2\Z)^k) \rightarrow L_p(\T)$ is an injective linear map taking characters to characters.  Then
\begin{equation*}
\|T\|_{p \rightarrow p} \geq \exp\left(ck(2-p)\right)
\end{equation*}
for some absolute $c>0$.
\end{corollary}
\begin{proof}
For $1 \leq p < 2$ take $G=(\Z/2\Z)^k$, $H=\T$, $\Lambda=\wh{G}$ and $m=2$.  In this case we see that $\|\pi_\Lambda\|=1$ and since $T$ is injective taking characters to characters and multiplication by $2$ is injective on $T(\Lambda)$ we have $|T(\Lambda)| =2^k$.  The claimed bound follows. 
\end{proof}
The same argument works for any finite Abelian group of exponent $r$ with $k$ invariant factors in place of $(\Z/2\Z)^k$ except that $c$ now depends on $r$.  This in turn means that the next proof can be adapted to give Theorem \ref{thm.pel} for any bounded Vilenkin system in place of the Walsh functions.

\begin{proof}[Proof of Theorem \ref{thm.pel}]
Suppose that $T:L_p[0,1) \rightarrow L_p[0,1)$ is an isomorphism extending some bijection $\sigma$ from the Walsh functions onto the trigonometric functions.  If $2<p<\infty$ then $T^*:L_{p'}[0,1) \rightarrow L_{p'}[0,1)$ is an isomorphism by \cite[I.A.12, I.A.13, I.A.14]{woj::} under the usual identification of the duals of Lebesgue spaces \cite[\S I.B.4]{woj::}.

For finite $J \subset \N_0$ and $z \in \Z$
\begin{equation*}
\langle T^*w_J,e_z\rangle = \langle w_J,Te_z\rangle = \langle w_J,w_{\sigma^{-1}(z)}\rangle = \begin{cases} 1 & \text{ if } z=\sigma(J)\\ 0 &\text{ otherwise.}\end{cases}
\end{equation*}
Since $(e_z)_{z \in\Z}$ is an orthonormal basis for $L_2[0,1)$ we see that $T^*$ maps Walsh functions onto the trigonometric functions.  It follows that $(T^*)^{-1}$ is an isomorphism $L_{p'}[0,1) \rightarrow L_{p'}[0,1)$ taking Walsh functions to trigonometric functions.  We conclude that we may assume $1<p<2$.

Define the (rather clumsily written) maps
\begin{equation*}
R:L_p[0,1) \rightarrow L_p(\T); f \mapsto \left(t+\Z\mapsto f\left(\sum_{x \in [0,1)}{x1_{[x \in t+\Z]}}\right)\right),
\end{equation*}
and
\begin{equation*}
S:L_p[0,1) \rightarrow L_p(\D); f \mapsto \left(t\mapsto f\left(\sum_{i=1}^\infty{2^{-i}1_{[t_i=1+2\Z]}} \right)\right).
\end{equation*}
These are isometric isomorphisms extending a bijection between the trigonometric functions and the characters on $\T$, and the other between the Walsh functions and the characters on $\D$.  (Formally this is the relationship between trigonometric and Walsh functions, and characters of compact Abelian groups which was mentioned in the introduction.)

Let $k \in \N$ be large enough that the lower bound in Corollary \ref{cor.1} is strictly larger than $\|T\|_{p \rightarrow p}$.  Write $\pi:\D \rightarrow (\Z/2\Z)^k$ for the projection onto the first $k$ coordinates of $\D$ so that the map $P:L_p((\Z/2\Z)^k)\rightarrow L_p(\D); f \mapsto f \circ \pi$ is an isometric isomorphism onto its image taking characters to characters.  Then the map $RTS^{-1}P:L_p((\Z/2\Z)^k) \rightarrow L_p(\T)$ is an injective linear map taking characters to characters of norm at most $\|T\|_{p \rightarrow p}$.  This contradicts the choice of $k$ after application of Corollary \ref{cor.1}.
\end{proof}

\section{Proof of Theorem \ref{thm.main}}

To make use of the hypotheses on $G$ and $H$ we shall need the following lemma.
\begin{lemma}\label{lem.mui}
Suppose that $G$ is a compact Abelian torsion group and $H$ is a finite Abelian group such that there is no continuous surjective homomorphism $G \rightarrow H^\omega$.  Then there is some $m\in \N$ such that $m\cdot \wh{G}$ is finite and $m\cdot \wh{H} \neq \{0_{\wh{H}}\}$.
\end{lemma}
\begin{proof}
Since $G$ is a torsion group it does not contain any elements of infinite order and so \cite[\S2.5.3]{rud::1} is of bounded order.  It follows \cite[\S2.5.4]{rud::1} that $\wh{G}$ is (a discrete Abelian group) of bounded order and so by \cite[B8]{rud::1} there is a set $\mathcal{C}$ of cyclic subgroups of $\wh{G}$ such that $\wh{G}= \bigoplus_{C \in \mathcal{C}}{C}$.  By the Fundamental Theorem of Arithmetic and the Chinese Remainder Theorem every cyclic group can be written as a direct sum of cyclic groups of prime-power order and so we may suppose that every $C \in \mathcal{C}$ has order a power of a prime.

Write $D$ for the set of primes $p$ such that there are infinitely many $C \in \mathcal{C}$ with order divisible by $p$.  Since $\wh{G}$ has bounded order the set $D$ is finite.  For each $p \in D$ write $e_p$ for the largest natural number such that there are infinitely many $C \in \mathcal{C}$ of order $p^{e_p}$ which exists since $\wh{G}$ has bounded order.  

Write $m:=\prod_{p \in D}{p^{e_p}}$ and note that $m\cdot \wh{G}$ is finite.  For each $p \in D$ let $(C_{i,p})_{i \in \N}$ be distinct elements of $\mathcal{C}$ of order $p^{e_p}$, which exist by hypothesis, and for each $i \in \N$ let $\Gamma_i:=\bigoplus_{p\in D}{C_{i,p}}$ which is a cyclic group of order $m$ by the Chinese Remainder Theorem.

Since $H$ is finite, $\wh{H}$ is finite and by the structure theorem for finite Abelian groups (invariant factor form) there are integers $1 < d_1 \divides d_2 \divides \cdots \divides d_r$ and subgroups $\Lambda_1,\dots,\Lambda_r$ such that $\Lambda_i$ is cyclic of order $d_i$ and $\wh{H}=\Lambda_1\oplus \cdots \oplus \Lambda_r$.

Suppose that $d_r \divides m$.  Then $d_j \divides m$ for all $1 \leq j \leq r$ and so for each $i\in \N$ there is an injective homomorphism
\begin{equation*}
\phi_i:\Lambda_1\oplus \cdots \oplus \Lambda_r \rightarrow \Gamma_{1+r(i-1)}\oplus \Gamma_{2+r(i-1)}\oplus \cdots \oplus \Gamma_{r + r(i-1)}.
\end{equation*}
Write $\wh{H}^{\N}$ for the set of functions $f:\N \rightarrow \wh{H}$ with finite support; by \cite[2.2.3]{rud::1} this is homeomorphically isomorphic to $(H^\omega)^\wedge$.

The injections above give rise to an injective homomorphism
\begin{equation*}
\wh{H}^{\N} \rightarrow \wh{G}; f \mapsto \sum_{i \in \supp f}{\phi_i(f(i))}.
\end{equation*}
This in turn induces a continuous surjective homomorphism $(\wh{G})^\wedge \rightarrow \left(\wh{H}^{\N}\right)^{\wedge}$.  By Pontryagin duality  \cite[\S1.7.2]{rud::1} the first group is homeomorphically isomorphic to $G$ and the second to $H^\omega$.  The lemma is proved.
\end{proof}
As in \cite{czuwoj::} we shall use a quantitative version of Cohen's Idempotent Theorem which we now record.
\begin{theorem}[{\cite[Theorem 1.1]{san::12}}]\label{thm.idem}
Suppose that $M \geq 1$.  Then for all finite Abelian groups $G$ and functions $f:G \rightarrow \Z$ with $\|f\|_{A(G)} \leq M$ there is some natural number $l$, integers $z_1,\dots,z_l$, and cosets $W_1,\dots,W_l$ of (possibly different) subgroups of $G$, such that
\begin{equation*}
f=z_11_{W_1}+\cdots +z_l1_{W_l} \text{ and } \|z\|_{\ell_1} \leq \exp\left(M^{4+o(1)}\right).
\end{equation*}
\end{theorem}
With these results in hand we are ready to prove the result essentially following the arguments of Rudin.
\begin{proof}[Proof of Theorem \ref{thm.main}]
To start the argument we apply Lemma \ref{lem.mui} to $G$ and $H$ to fix $m=m_{G,H} \in \N$ such that $m\cdot \wh{G}$ is finite, say of size $M$, and $m\cdot \wh{H} \neq \{0_{\wh{H}}\}$.  

Suppose that $T:L_1(H^n) \rightarrow L_1(G)$ is a continuous injective algebra homomorphism.  Following \cite[\S4.1.1]{rud::1} there is a set $S \subset \wh{G}$ and a map $\alpha:S \rightarrow \wh{H}^n$ such that for all $f \in L_1(H^n)$ we have $T(f)^\wedge(\gamma)=\wh{f}(\alpha(\gamma))$ for all $\gamma \in S$ and $T(f)^\wedge(\gamma)=0$ otherwise.  By \cite[\S4.4.3]{rud::1} the graph $\Gamma_0:=\{(\lambda,\alpha(\lambda)): \lambda \in S\} \subset \wh{G} \otimes \wh{H}^n$ has $1_{\Gamma_0} \in B(\wh{G} \otimes \wh{H}^n)$ and $\|1_{\Gamma_0}\|_{B(\wh{G} \otimes \wh{H}^n)} \leq \|T\|$.  (This inequality is not in the statement but is established in the last line of the proof of \cite[\S4.4.3]{rud::1}.)

Since $T$ is injective, $\alpha$ is a surjection and so there is a finite set $F \subset S \subset \wh{G}$ such that $\alpha(F)=\wh{H}^n$.  Since $\wh{G}$ has bounded order (as before, by \cite[\S2.5.3]{rud::1} and \cite[\S2.5.4]{rud::1} since $G$ is a compact Abelian torsion group) we see that $V:=\langle F\rangle$ is finite.  Write $\Gamma$ for the graph of $\alpha|_V$ which is surjective by design and has $\Gamma=\Gamma_0 \cap (V \otimes \wh{H}^n)$.  This is finite so
\begin{equation*}
\|1_{\Gamma}\|_{A(V \otimes \wh{H}^n)}=\|1_{\Gamma}\|_{B(\wh{G} \otimes \wh{H}^n)}=\|1_{V \otimes \wh{H}^n}1_{\Gamma_0}\|_{B(\wh{G}\otimes \wh{H}^n)} \leq \|1_{V \otimes \wh{H}^n}\|_{B(\wh{G}\otimes \wh{H}^n)}\|1_{\Gamma_0}\|_{B(\wh{G}\otimes \wh{H}^n)} \leq \|T\|,
\end{equation*}
and hence by the Quantitative Idempotent Theorem there is a natural $l$, integers $z_1,\dots,z_l$, and cosets $\Gamma_1,\dots,\Gamma_l$ of (possibly different) subgroups of $V \otimes \wh{H}^n$ such that
\begin{equation*}
1_\Gamma =z_11_{\Gamma_1}+\cdots + z_l1_{\Gamma_l} \text{ and } l \leq \|z\|_{\ell_1} \leq \exp\left(\|T\|^{4+o(1)}\right).
\end{equation*}
Recall the choice of $m$ from the very start.  If $(\lambda,\gamma) \in m\cdot \Gamma$ then there is $(\lambda',\gamma') \in \Gamma$ such that $\lambda =m\cdot \lambda'$ and $\gamma=m\cdot \gamma'$.  By the above equality we must have some $1 \leq i \leq l$ such that $1_{\Gamma_i}(\lambda',\gamma') \neq 0$ and hence $1_{m\cdot \Gamma_i}(\lambda,\gamma)\neq 0$.  We conclude that
\begin{equation*}
1_{m\cdot \Gamma} \leq 1_{m\cdot \Gamma_1}+\cdots + 1_{m\cdot \Gamma_l}.
\end{equation*}
The set $F$, which generates $V$, has been chosen so that $\alpha(F)=\wh{H}^n$ and since $|m\cdot \wh{H}|\geq 2$ we have
\begin{equation*}
2^n \leq \sum_{\gamma \in (m\cdot \wh{H})^n}{\sum_{\lambda \in m\cdot V}{1_{m\cdot \Gamma}(\lambda,\gamma)}} \leq l\max_{1 \leq i \leq l}{\sum_{\gamma \in (m\cdot \wh{H})^n}{\sum_{\lambda \in m\cdot V}{1_{m\cdot \Gamma_i}(\lambda,\gamma)}}}.
\end{equation*}
Let $1 \leq i\leq l$ be such that
\begin{equation*}
\exp(-\|T\|^{4+o(1)})2^n \leq \sum_{\gamma \in (m\cdot \wh{H})^n}{\sum_{\lambda \in m\cdot V}{1_{m\cdot \Gamma_i}(\lambda,\gamma)}} = |m\cdot \Gamma_i|.
\end{equation*}

Finally, apply Proposition \ref{prop.k} with function $\alpha$, parameter $m$, and the coset $\Gamma_i$ to get
\begin{equation*}
|\Gamma_i|=|\Gamma_i+\Gamma_i| \geq \left(\frac{|m\cdot \Gamma_i|}{|m\cdot \wh{G}|}\right)^{\frac{1}{m+1}}|\Gamma_i| \geq \left(\frac{\exp(-\|T\|^{4+o(1)})2^n}{M}\right)^{\frac{1}{m+1}}|\Gamma_i|.
\end{equation*}
The main part of the result is proved.

For the last part suppose that there were a continuous injective algebra homomorphism $R:L_1(H^\omega) \rightarrow L_1(G)$, and let $n \geq N_{G,H}$ be a natural such that the lower bound in the main part of the theorem is strictly larger than $\|R\|$.  Projection onto the first $n$ coordinates is a continuous surjective homomorphism $H^\omega \rightarrow H^n$, and hence there is a norm $1$ injective algebra homomorphism $S:L_1(H^n) \rightarrow L_1(H^\omega)$.  But then $RS:L_1(H^n) \rightarrow L_1(G)$ is an injective algebra homomorphism of norm $\|RS\| \leq \|R\|$, which contradicts how $n$ has been chosen.  The result is proved.
\end{proof}

\section*{Acknowledgements}

The author should very much like to thank the referees for their careful reading of the paper and in particular, for identifying an error in an early version of the proof of Theorem \ref{thm.main}.  We take a moment to explain this now.

In the earlier version of the paper (available on the arXiv at \cite{san::51}) the Balog-Szemer{\'e}di-Gowers Theorem was used in place of the Quantitative Idempotent Theorem in the proof of Theorem \ref{thm.main} and a stronger bound of the form
\begin{equation}\label{eqn.jhg}
\|T\| \geq (1+c_{G,H})^n \text{ for all }n \geq N_{G,H}
\end{equation}
with constants $c_{G,H}>0$ and $N_{G,H}\in \N$ depending only on $G$ and $H$, was claimed.  The proof was incorrect because while the Balog-Szemer{\'e}di-Gowers Theorem guarantees a large part of the graph of a function with large additive energy is structured, it might be that that part has a very small projection onto the codomain of the function.  It seems likely that if the following conjecture were true then bounds of the form (\ref{eqn.jhg}) could be recovered.
\begin{conjecture}
Suppose that $G$ is a finite Abelian group and $A \subset G$ has $\|1_A\|_{A(G)} \leq M$.  Then there is $l=M^{O(1)}$ and sets $A_1,\dots,A_l$ partitioning $A$ such that
\begin{equation*}
|A_i+A_i| \leq M^{O(1)}|A_i| \text{ for all }1 \leq i \leq l.
\end{equation*}
\end{conjecture}

\bibliographystyle{halpha}

\bibliography{references}

\begin{thebibliography}{Gow98}
\expandafter\ifx\csname url\endcsname\relax
  \def\url#1{\texttt{#1}}\fi
\expandafter\ifx\csname doi\endcsname\relax
  \def\doi#1{\burlalt{doi:#1}{http://dx.doi.org/#1}}\fi
\expandafter\ifx\csname urlprefix\endcsname\relax\def\urlprefix{URL }\fi
\expandafter\ifx\csname href\endcsname\relax
  \def\href#1#2{#2}\fi
\expandafter\ifx\csname burlalt\endcsname\relax
  \def\burlalt#1#2{\href{#2}{#1}}\fi

\bibitem[Buk08]{buk::0}
B.~Bukh.
\newblock Sums of dilates.
\newblock {\em Combin. Probab. Comput.}, 17(5):627--639, 2008,
  \burlalt{arXiv:0711.1610}{http://arxiv.org/abs/arXiv:0711.1610}.
\newblock \doi{10.1017/S096354830800919X}.

\bibitem[CW13]{czuwoj::}
A.~Czuron and M.~Wojciechowski.
\newblock {On the isomorphisms of Fourier algebras of finite Abelian groups}.
\newblock {\em ArXiv e-prints}, June 2013,
  \burlalt{arXiv:1306.1480}{http://arxiv.org/abs/arXiv:1306.1480}.

\bibitem[Gow98]{gow::4}
W.~T. Gowers.
\newblock A new proof of {S}zemer\'edi's theorem for arithmetic progressions of
  length four.
\newblock {\em Geom. Funct. Anal.}, 8(3):529--551, 1998.
\newblock \doi{10.1007/s000390050065}.

\bibitem[HW03]{hinwen::0}
A.~Hinrichs and J.~Wenzel.
\newblock On the non-equivalence of rearranged {W}alsh and trigonometric
  systems in {$L_p$}.
\newblock {\em Studia Math.}, 159(3):435--451, 2003,
  \burlalt{arXiv:math/0308091}{http://arxiv.org/abs/arXiv:math/0308091}.
\newblock \doi{10.4064/sm159-3-7}.

\bibitem[Pe{\l}06]{pel::1}
A.~Pe{\l}czy\'{n}ski.
\newblock Selected problems on the structure of complemented subspaces of
  {B}anach spaces.
\newblock In {\em Methods in {B}anach space theory}, volume 337 of {\em London
  Math. Soc. Lecture Note Ser.}, pages 341--354. Cambridge Univ. Press,
  Cambridge, 2006.
\newblock \doi{10.1017/CBO9780511721366.018}.

\bibitem[Pet12]{pet::}
G.~Petridis.
\newblock New proofs of {P}l{\"u}nnecke-type estimates for product sets in
  groups.
\newblock {\em Combinatorica}, 32(6):721--733, 2012,
  \burlalt{arXiv:1101.3507}{http://arxiv.org/abs/arXiv:1101.3507}.
\newblock \doi{10.1007/s00493-012-2818-5}.

\bibitem[Pie99]{pie::}
A.~Pietsch.
\newblock What is ``local theory of {B}anach spaces''?
\newblock {\em Studia Math.}, 135(3):273--298, 1999.
\newblock \doi{10.4064/sm-135-3-273-298}.

\bibitem[Rud90]{rud::1}
W.~Rudin.
\newblock {\em Fourier analysis on groups}.
\newblock Wiley Classics Library. John Wiley \& Sons Inc., New York, 1990.
\newblock \doi{10.1002/9781118165621}.
\newblock Reprint of the 1962 original, A Wiley-Interscience Publication.

\bibitem[San19]{san::51}
T.~Sanders.
\newblock On inequivalences of sequences of characters.
\newblock {\em ArXiv e-prints}, 2019,
  \burlalt{arXiv:1901.03109}{http://arxiv.org/abs/arXiv:1901.03109}.

\bibitem[San20]{san::12}
T.~Sanders.
\newblock Bounds in {C}ohen's idempotent theorem.
\newblock {\em Journal of Fourier Analysis and Applications}, 26(2):25, 2020,
  \burlalt{arXiv:1610.07092}{http://arxiv.org/abs/arXiv:1610.07092}.
\newblock \doi{10.1007/s00041-020-09732-y}.

\bibitem[Sch15]{sch::7}
T.~Schoen.
\newblock New bounds in {B}alog-{S}zemer{\'e}di-{G}owers theorem.
\newblock {\em Combinatorica}, 35(6):695--701, 2015.
\newblock \doi{10.1007/s00493-014-3077-4}.

\bibitem[TV06]{taovu::}
T.~C. Tao and V.~H. Vu.
\newblock {\em Additive combinatorics}, volume 105 of {\em Cambridge Studies in
  Advanced Mathematics}.
\newblock Cambridge University Press, Cambridge, 2006.
\newblock \doi{10.1017/CBO9780511755149}.

\bibitem[Woj91]{woj::}
P.~Wojtaszczyk.
\newblock {\em Banach spaces for analysts}, volume~25 of {\em Cambridge Studies
  in Advanced Mathematics}.
\newblock Cambridge University Press, Cambridge, 1991.
\newblock \doi{10.1017/CBO9780511608735}.

\bibitem[Woj00]{woj::1}
P.~Wojtaszczyk.
\newblock Non-similarity of {W}alsh and trigonometric systems.
\newblock {\em Studia Math.}, 142(2):171--185, 2000.
\newblock \doi{10.4064/sm-142-2-171-185}.

\bibitem[Woj11]{woj::0}
M.~Wojciechowski.
\newblock The non-equivalence between the trigonometric system and the system
  of functions with pointwise restrictions on values in the uniform and {$L^1$}
  norms.
\newblock {\em Math. Proc. Cambridge Philos. Soc.}, 150(3):561--571, 2011.
\newblock \doi{10.1017/S0305004111000065}.

\bibitem[You76]{you::0}
W.~S. Young.
\newblock A note on {W}alsh-{F}ourier series.
\newblock {\em Proc. Amer. Math. Soc.}, 59(2):305--310, 1976.
\newblock \doi{10.2307/2041490}.

\bibitem[Zyg02]{zyg::}
A.~Zygmund.
\newblock {\em Trigonometric series. {V}ol. {I}, {II}}.
\newblock Cambridge Mathematical Library. Cambridge University Press,
  Cambridge, third edition, 2002.
\newblock \doi{10.1017/CBO9781316036587}.
\newblock With a foreword by Robert A. Fefferman.

\end{thebibliography}

\end{document}